\documentclass[a4paper]{article}

\usepackage{amsfonts,amsmath,amssymb,amsthm}
\usepackage[T1]{fontenc}

\def\C{\mathbb{C}}
\def\H{\mathbb{H}}
\def\N{\mathbb{N}}

\def\Z{\mathbb{Z}}

\def\CP{\mathbb{P}_{\C}}
\def\reg{\text{reg.}}

\def\={=&\:}
\def\+{+&\:}
\def\-{-&\:}

\newcommand{\cyc}{\text{cyclic}}

\newcommand{\eps}{\varepsilon}
\newcommand{\rechts}{\rightarrow}
\newcommand{\cdim}{\text{dim}_{\C}}

\newcommand{\1}{\langle 1\rangle}
\newcommand{\T}{\langle T\rangle}

\newcommand{\tx}{\tilde{x}}
\newcommand{\ty}{\tilde{y}}

\newcommand{\nn}{\nonumber}

\newtheorem{theorem}{Theorem}
\newtheorem{lemma}{Lemma}
\newtheorem{cor}{Corollary}
\newtheorem{Def}{Definition}
\newtheorem{remark}{Remark}
\newtheorem{example}{Example}

\makeatletter
\renewcommand\@biblabel[1]{#1.}
\makeatother

\begin{document}

\title{CFT on Riemann Surfaces of genus $g\geq 1$\footnote{The paper is published as \textit{Virasoro correlation functions on hyperelliptic Riemann surfaces}, 
Lett. Math. Phys. \textbf{103.7} (2013), 701--728 (DOI 10.1007/s11005-013-0608-7).}}

\author{Marianne Leitner\\
School of Mathematics,\\
Trinity College, Dublin 2, Ireland\\
and\\
Dublin Institute for Advanced Studies,\\
School of Theoretical Physics,\\
10 Burlington Road, Dublin 4, Ireland\\
\textit{leitner@stp.dias.ie}
}

\maketitle
\begin{abstract}
$N$-point functions of holomorphic fields in conformal field theories can be calculated by methods from algebraic geometry.
We establish explicit formulas for the $2$-point function of the Virasoro field on hyperelliptic Riemann surfaces of genus $g\geq 1$.
Virasoro $N$-point functions for higher $N$ are obtained inductively, and we show that they have a nice graph representation. 
We discuss the $3$-point function with application to the $(2,5)$ minimal model.
\end{abstract}

\maketitle



\tableofcontents

\section{Introduction}

Quantum field theories are a major challenge for mathematicians. 
Apart from cases without interaction, the theories best understood at present are conformally invariant 
and do not contain massive particles.

Conformal Field Theories (CFTs) can be defined over arbitrary Riemann surfaces.
A theory is considered to be solved once all of its $N$-point functions are known.
We restrict our consideration to meromorphic CFTs \cite{Goddard: 1989} which are defined by holomorphic fields,
and a rather specific class of Riemann surfaces.

The case of the Riemann sphere $X_0$ is easy, 
and for the torus $X_1$, one can use the standard tools of doubly periodic and modular functions 
(\cite{Zhu:1996},\cite{DiFranc:1997} and more recently, e.g.\ \cite{DG:2009},\cite{NK:2009}).
The case $g>1$ is technically more demanding, however. 
Some progress has been made in the Vertex Operator Algebra (VOA) formalism
by sewing surfaces of lower genus.
There is no canonical way to do this and two different sewing procedures have been explored.
Explicit formulas could be extablished for the genus two $N$-point functions for the free bosonic Heisenberg VOA and its modules 
(\cite{MT:2010}, \cite{MT:2011}),
and for the free fermion vertex operator superalgebra \cite{TZ:2011}.

Instead, quantum field theory on a compact Riemann surface of any genus can be approached differently 
using methods from algebraic geometry (\cite{R:1989-91}, \cite{TUY:1989}, \cite{F-BZ:2004}) and complex analysis.
$N$-point functions of holomorphic fields are meromorphic functions.
That is, they are determined by their poles and their respective behaviour at infinity.
By compactness of $X_g$, these functions are rational.

The present paper establishes explicit formulas for the $2$-point functions of the Virasoro field
over hyperelliptic genus-$g$ Riemann surfaces $X_g$, where $g\geq 1$.
$N$-point functions for $N\geq 3$ are obtained inductively from these,
up to a finite number of parameters which in general cannot be determined by the methods presented in this paper.
In comparison, the formulas given by the work of Mason, Tuite, and Zuevsky 
determine all constants, but are given in terms of infinite series.

We show that the $N$-point functions can be written in terms of a list of oriented graphs.  
For $g=1$ the result reduces to a formula which is very similar to eq.\ (3.19) in \cite{HT:2011}.
The method we used is essentially the one developed in \cite{HT:2011}
though it was found independently.

Although we deal with the Virasoro field, our method applies to more general holomorphic fields.

\section{Notations}
 
For any $k\geq0$ and any rational function $R(x)$ of $x$ with Laurent expansion
\begin{displaymath}
R(x)=\sum_{i\in\Z}a_ix^{i} 
\end{displaymath}
for large $|x|$, we define the polynomial
\begin{align}\label{projection of a rational function onto a part of bounded degree}
[R(x)]_{>k}
:=\sum_{i>k}a_ix^i\:. 
\end{align}

\section{Rational coordinates}

Let $X_1$ be a compact Riemann surface of genus $g=1$.
Such a manifold is biholomorphic to the torus $\C/\Lambda$ (with the induced complex structure),
for the lattice $\Lambda$ spanned over $\Z$ by $1$ and some $\tau\in\H^+$,
unique up to an $SL(2,\Z)$ transformation. 
Here $\H^+$ denotes the upper complex half plane.
We denote by $z$ the local coordinate on $X_1$
and by $z_1,\ldots z_N$ the corresponding variables of the $N$-point functions on $X_1$ \cite{DiFranc:1997}.
The latter are elements of $\C(\wp(z_1),\wp'(z_1),\ldots,\wp(z_N),\wp'(z_N))/J$, 
where $\wp$ is the Weierstrass function associated to $\Lambda$,
$\wp'=\partial\wp/\partial z$, 
and $J$ is the ideal defined by $y^2=4(x^3-15G_4x-35G_6)$, where 
\begin{align}\label{definition of x and y for n=3}
x=\wp(z),\quad y=\wp'(z),\quad(\tau\:\text{fixed}). 
\end{align}
$G_{2k}$ for $k\geq 2$ are the holomorphic Eisenstein series.
Instead of $z$ we shall work with the pair of complex coordinates $x,y$ defined by (\ref{definition of x and y for n=3}).
We compactify $X_1$ by including the point $x=\infty$ (corresponding to $z=0$ mod $\Lambda$),
and view $x$ as a holomorphic function on $\C/\Lambda$ with values in $\CP^1$.
$y=\wp'(z)$ defines a double cover of $\CP^1$.
$N$-point functions can be expressed in terms of $\wp(z_1),\wp'(z_1),\ldots$, 
or equivalently as rational functions of $x_1,y_1,\ldots,x_N,y_N$.
The latter possibility generalizes much more easily to higher genus.
Instead, one can try to work with the Jacobian of the curve and the corresponding theta functions.
This would generalize to arbitrary curves, but it is unknown for which conformal field theories this is possible. 

If $g>1$, one can write $X_g$ as the quotient of $\H^+$ by a Fuchsian group, 
but working with a corresponding local coordinate $z$ becomes difficult 
(e.g.\ \cite{FK:1980}, and more recently \cite{Maskit:1999}). 
We shall consider \textit{hyperelliptic} Riemann surfaces $X_g$ only, where $g\geq 1$.
Such surfaces are defined by
\begin{align*}
X_g:\quad y^2=p(x),
\end{align*}
where $p$ is a polynomial of degree $n=2g+1$ (the case $n=2g+2$ is equivalent
and differs from the former by a rational transformation of $\C$ only).
We assume $p$ has no multiple zeros so that $X_g$ is \textit{regular}.    
Locally we will work with one complex coordinate, either $x$ or $y$.
A coordinate on $X_g$ out of the set of $x$ and $y$ is \textit{locally an admissible coordinate}
if it defines a chart.
$x$ is an admissible coordinate away from the ramification points (where $p=0$),
whereas $y$ is admissible away from the locus where $p'=0$.

\section{The Virasoro OPE}\label{section: The Virasoro OPE}

\subsection{The vector bundle of holomorphic fields}\label{introduction of vector bundle of fields}

For any Riemann surface $X$, the holomorphic fields of a meromorphic CFT on $X$ form a vector bundle $\mathcal{F}$ over $X$ 
whose trivialisation on parametrized open sets is canonical.
More specifically, let $(U,z)$ be a chart on $X$: The holomorphic map
\begin{align*}
U\:\overset{z}{\rechts}\:\C\: 
\end{align*}
is called a coordinate on $U$, and $U$ will be referred to as a \textit{coordinate patch}. 
We postulate that $(U,z)$ induces a trivialization
\begin{align*}
\mathcal{F}|_U\:\overset{z^*}{\rechts}\:F\times U\:.
\end{align*}
Here the fiber $F$ is the infinite dimensional complex vector space of holomorphic fields.
For $U'\subseteq U$, the trivialization corresponding to $(U',z)$ is induced by the one for $(U,z)$.
For any coordinate patch $U$ with coordinate $z$, elements of $\mathcal{F}|_U$ can be written as 
\begin{align*}
\varphi_z(u)
=(z^*)^{-1}(\varphi\times\{u\})\:,
\end{align*}
with $\varphi\in F$, $u\in U$.
Abusing notations, we shall simply write $\varphi(z)$ where we actually mean $\varphi_z(u)$.
(This will entail notations like $\hat{\varphi}(\hat{z})$ instead of $\varphi_{\hat{z}}(u)$ etc.). 
Thus an isomorphism between two coordinate patches on $X$ induces an isomorphism between the corresponding fields.  
On $\CP^1$ the bundle $\mathcal{F}$ splits into line bundles. 
The corresponding Chern numbers can be established to be the (holomorphic) dimensions of the fields.
For $\C\subset\CP^1$, the splitting of $\mathcal{F}$ induces a $\Z$ grading of the fiber $F$. 
Thus to every (nonzero) field $\varphi\in F$ there is associated the \textit{(holomorphic) dimension} $h(\varphi)$ of $\varphi$. 
We assume that 
\begin{align}\label{non-negativity of holomorphic dimension}
h(\varphi)\geq 0\:,\quad\forall\:\varphi\in F\:, 
\end{align}
so that
\begin{displaymath}
F=\bigoplus_{h\in\N_0}\:F(h)\:,
\end{displaymath}
where $F(0)\cong\C$ is spanned by the identity field $1$, 
and we assume that for any $h\in\N_0$, the dimension of $F(h)$ is finite.

We postulate that for any Riemann surface $X$ and any $u\in X$, 
the ascending filtration of the fiber $\mathcal{F}_u$ (of $\mathcal{F}$ in $u$) associated to the grading
does not depend on the choice of the coordinate.
Since in a conformal field theory fields of finite dimension only are considered,
it is sufficient to deal with finite sums.

It may be useful to compare our formalism to the approach by P.\ Goddard \cite{Goddard: 1989} 
where only the case $g=0$ is discussed in detail.
Goddard interprets $F$ as a dense subspace of a space of states $\mathcal{H}$ 
using the field-state correspondence. 
He works on $\C\subset\CP^1$. 
In our notation this corresponds to the identity map $\text{id}:U\rechts\C$.
Our field $\psi_{id}(z)$ is Goddard's $V(\psi,z)$.
We will not use the field-state correspondence and reserve the word \textit{state} for something different.
Our notion of state on a Riemann surface $X$ is a map $\langle\quad\rangle$ from products of fields 
$\Psi=\psi_{z_1}(p_1)\otimes\ldots\otimes\psi_{z_N}(p_N)$
to numbers $\langle\Psi\rangle\in\C$, in analogy to the language of operator algebra theory.  
We will not use the interpretation of fields as operators, however, 
since the necessary ordering is unnatural for $g>1$.  

\subsection{Meromorphic conformal field theories}

Let $X_g$ be a connected Riemann surface of genus $g\geq 1$
(when the genus is fixed, we shall simply refer to $X_g$ as $X$).
We don't give a complete definition of a \textit{meromorphic conformal field theory} here,
but the most important properties are as follows \cite{N:2009}:
\begin{enumerate}
\item\label{symmetry}
For $i=1,2$, let $X_i$ be a Riemann surface and let $\mathcal{F}_i$ be a rank $r_i$ vector bundle over $X_i$.
Let $p_i^*\mathcal{F}_i$ be the pullback bundle of $\mathcal{F}_i$ by the morphism $p_i:X_1\times X_2\rechts X_i$.
Let
\begin{align*}
\mathcal{F}_1\boxtimes\mathcal{F}_2
&:=p_1^*\mathcal{F}_1\otimes p_2^*\mathcal{F}_2 
\end{align*}
be the rank $r_1r_2$ vector bundle whose fiber at $(z_1,z_2)\in X_1\times X_2$ 
is $\mathcal{F}_{1,z_1}\otimes\mathcal{F}_{2,z_2}$.
We are now in position to define $N$-point functions for bosonic fields.
Let $\mathcal{F}$ be the vector bundle introduced in section \ref{introduction of vector bundle of fields}. 
A \textit{state} on $X$ is a multilinear map
\begin{align*}
\langle\quad\rangle:\quad\quad
S_*(\mathcal{F})\quad
&\rightarrow\quad\C,
\end{align*}
where $S_*(\mathcal{F})$ denotes the restriction of the symmetric algebra $S(\mathcal{F})$
to fibers away from the partial diagonals
\begin{align*}
\Delta_N
:=\{(z_1\ldots,z_N)\in X^N|\:z_i=z_j,\:\text{for some}\:i\not=j\},
\end{align*}
for any $N\in\N$. 
For ease of notations, 
when writing $\otimes$ and $\boxtimes$ we shall in the following actually mean the respective symmetrized product.

Locally, over  any $U^N\subseteq X^N\setminus\Delta_N$ such that $(U,z)$ defines a chart on $X$,
a state is the data for any $N\in\N$ 
of an $N$-linear holomorphic map 
\begin{align*}
\langle\quad\rangle:\quad\quad\quad\quad
F^{\otimes N}\times U^N\quad
&\rightarrow\quad\C\\
(\varphi_1,z_1)\boxtimes\ldots\boxtimes(\varphi_N,z_N)
&\mapsto
\langle\varphi_1(z_1)\otimes\ldots\otimes\varphi_N(z_N)\rangle
\end{align*}
satisfying the following conditions:
\begin{enumerate}
\item $\langle\quad\rangle$ is compatible with the Operator Product Expansion (OPE).
(The OPE is defined below in point \ref{definition of OPE}, and the compatibility condition is explained in point \ref{compatibility of state with OPE}.) 
\item
For $\varphi_1=1$, we have
\begin{displaymath}
\langle 1\otimes\varphi_2(z_2)\otimes\ldots\otimes\varphi_N(z_N)\rangle=\langle\varphi_2(z_2)\otimes\ldots\otimes\varphi_N(z_N)\rangle\:. 
\end{displaymath}
\end{enumerate}

\begin{remark}
In standard physics' notation the symbol for the symmetric tensor product is omitted.
We shall adopt this notation 
and write $$\langle\varphi_1(z_1)\ldots\varphi_N(z_N)\rangle$$ instead of $\langle\varphi_1(z_1)\otimes\ldots\otimes\varphi_N(z_N)\rangle$
but keep in mind that each $z_i$ lies in an individual copy of $U$
whence the $\varphi_i(z_i)$ are elements in different copies of $F$ and multiplication is meaningless. 
\end{remark}

Since each $\varphi_i$ is defined over $U$, we may view 
$\langle\varphi_1(z_1)\:\ldots\:\varphi_n(z_N)\rangle$ as a function of $(z_1,\ldots,z_N)\in U^N$.
We call it the $N$\textit{-point function} of the fields $\varphi_1,\ldots,\varphi_N$ over $U$.
For example, the zero-point function $\langle 1\rangle$ is a complex number.

\begin{remark}
One can make contact to the notion of $N$-point function used in \cite{Goddard: 1989}
by considering states for manifolds with boundary (see G.\ Segal's axioms) \cite{Segal:1989}.
\end{remark}

\item
Fields are understood by means of their $N$-point functions. 
A field $\phi$ is zero if all $N$-point functions involving $\phi$ vanish. 
That is, for any $N\in\N$, $N\geq 2$, and any set $\{\phi_2,\ldots\phi_N\}$ of fields,
\begin{displaymath}
\langle\phi(z_1)\:\phi_2(z_2)\:\ldots\:\phi_N(z_N)\rangle=0\:. 
\end{displaymath}

\item\label{definition of OPE}
We assume the existence of an OPE on $F$, i.e.\ for any $m\in\Z$ of a linear degree $m$ map
\begin{align*}
N_m:\quad F\otimes F\:\rechts\: F\:. 
\end{align*}
$N_m$ has degree $m$ if for any $\varphi_1,\varphi_2\in F$, 
$N_m(\varphi_1,\varphi_2)$ has holomorphic dimension 
\begin{displaymath}
m+h(\varphi_1)+h(\varphi_2)\:.
\end{displaymath}
Note that the degree condition is void when $N_m(\varphi_1,\varphi_2)$ is the zero field.

\begin{remark}
For $\varphi\in F$, the family of induced linear maps $N_m(\varphi,\:):F\rechts F$ indexed by $m\in\Z$
span a VOA (in particular a Lie algebra).
\end{remark}

\item\label{compatibility of state with OPE}
While fields and coordinates are local objects, states should contain global information. 
A state is said to be \textit{compatible with the OPE} if
for any $N\in\N$, $N\geq 2$, and whenever $\varphi_1,\ldots,\varphi_N$ are holomorphic fields over a coordinate patch $U\subset X$,
the corresponding $N$-point function has a Laurent series expansion in $z_1$ about $z_1=z_2$ given by
\begin{align*}
&\langle\varphi_1(z_1)\:\varphi_2(z_2)\ldots\:\varphi_N(z_N)\rangle\nn\\
&\quad\quad
=\sum_{m\geq m_0}(z_1-z_2)^m\langle N_m(\varphi_1,\varphi_2)(z_2)\:\varphi_3(z_3)\ldots\varphi_N(z_N)\rangle, 
\end{align*}
for some $m_0\in\Z$. 
Symbolically we write 
\begin{displaymath}
\varphi_1(z_1)\:\varphi_2(z_2)\mapsto
\sum_{m\geq m_0}(z_1-z_2)^m N_m(\varphi_1,\varphi_2)(z_2)\:. 
\end{displaymath}
This arrow defines the OPE of $\varphi_1,\varphi_2\in\mathcal{F}|_U$.  
We postulate that every OPE admits compatible states.

\begin{remark}
Physicists write an equality here. Recall however that $\otimes$ is understood on the l.h.s.
\end{remark}

\item
We have $N_m(\varphi,1)=0$ for $\varphi\in F$ and $m<0$.
Define the derivative of a field $\varphi$ by
\begin{displaymath}
\partial\varphi
:=N_1(\varphi,1)\:.  
\end{displaymath}
Equivalently, $\partial\varphi$ is defined by prescribing
\begin{displaymath}
\langle\partial\varphi(z)\varphi_2(z_2)\:\ldots\:\varphi_N(z_N)\rangle
:=\partial_z\langle\varphi(z)\:\varphi_2(z_2)\:\ldots\:\varphi_N(z_N)\rangle\:,
\end{displaymath}
for all $N$-point functions involving $\varphi$.

\item\label{existence of a Virasoro field} 
In conformal field theories, one demands the existence of a \textit{Virasoro field} $T\in F(2)$ 
which satisfies
\begin{displaymath}
N_{-1}(T,\varphi)=\partial\varphi\:,
\end{displaymath}
whenever $\varphi$ is a holomorphic section in $U\times F$. 
\end{enumerate}

The choice $\varphi_1=\varphi_2=T$ in point \ref{compatibility of state with OPE} yields the \textbf{Virasoro OPE}.
It is specified by the assumptions made in Section \ref{introduction of vector bundle of fields} 
and the properties \ref{symmetry}-\ref{existence of a Virasoro field} above:   

\begin{lemma}
In local coordinates $z$ and $w$, the Virasoro OPE has the form
\begin{align}\label{OPE}
T(z)T(w)
\mapsto
\frac{c/2}{(z-w)^4}\:1+\frac{1}{(z-w)^2}\left(T(z)+T(w)\right)+\Phi(w)+O(z-w)\:,
\end{align}
for some $c\in\C$. 
\end{lemma}

The constant $c$ is called the \textit{central charge} of the theory. 
Note that $$\Phi=N_0(T,T)-\frac{\partial^2T}{2}\:.$$

\begin{proof}
By assumption (\ref{non-negativity of holomorphic dimension}), all holomorphic fields have non-negative dimension, and $h(T)=2$.
This yields the lowest order term, since $F(0)$ is spanned by $1$.
Symmetry (point \ref{symmetry}) implies the existence of a field $\Omega$, of dimension $2$, such that
\begin{align*}
T(z)T(w)
\mapsto&\frac{\frac{c}{2}.1}{(z-w)^4}+\frac{\Omega(z)+\Omega(w)}{(z-w)^2}+O(1)\\
\=\frac{\frac{c}{2}.1}{(z-w)^4}+\frac{2\Omega(w)}{(z-w)^2}+\frac{\partial\Omega(w)}{(z-w)}+O(1)\:. 
\end{align*}
Thus $N_{-1}(T,T)=\partial\Omega$, from which (considering dimensions) we conclude $\Omega=T$.
\end{proof}

\begin{example}
A \textbf{Virasoro} model is minimal if it has only finitely many non-isomorphic irreducible representations of the VOA.
For the $(p,q)$ minimal model the number of such representations is (e.g., \cite{DiFranc:1997}, \cite{Blum:2009}) 
\begin{align*}
\frac{(p-1)(q-1)}{2}\:. 
\end{align*}
The $(2,5)$ minimal model has just two irreducible representations, 
the vacuum representation $\1$ for the lowest weight $h=0$, and one other for $h=-\frac{1}{5}$.
\end{example}

Let us recapitulate the behaviour of $T$ under coordinate transformations.

\begin{Def}
Given a holomorphic function $f$ (with non-vanishing first derivative $f'$), we define the \textbf{Schwarzian derivative} of $f$ by
\begin{align*}
S(f)
:=\frac{f'''}{f'}-\frac{3[f'']^2}{2[f']^2}\:.
\end{align*}
\end{Def}

The Schwarzian derivative $S$ has the following well-known properties:
\begin{enumerate}
\item\label{scale invariance of S} 
$S(\lambda f)=S(f)$, $\forall\:\lambda\in\C$, $f\in\mathcal{D}(S)$, the domain of $S$.
\item\label{S(lin.fract.)}
Suppose $f:P_{\C}^1\rechts P_{\C}^1$ is a linear fractional (M\" obius) transformation, 
\begin{align*}
f:\:z\mapsto f(z)=\frac{az+b}{cz+d}\:,
\quad\text{where}\:
\begin{pmatrix}
a&b\\
c&d
\end{pmatrix}
\in SL(2,\C). 
\end{align*}
Then $f\in\mathcal{D}(S)$, and $S(f)=0$. 
\item\label{chain rule for the Schwarzian}
Let $f,g\in\mathcal{D}(S)$ be such that $f\circ g$ is defined and lies in $\mathcal{D}(S)$.
Then 
\begin{align*}
S(f\circ g)
=\left[g'\right]^2S(f)\circ g+S(g)\:. 
\end{align*}
\end{enumerate}

\begin{remark}
Let $p,y\in\mathcal{D}(S)$ with $y^2=p(x)$. 
Then by the properties \ref{scale invariance of S} and \ref{chain rule for the Schwarzian} of the Schwarzian derivative,
\begin{align}\label{S(p)}
S(y)
=S(p)+\frac{3}{8}\left[\frac{p'}{p}\right]^2\:.
\end{align}
\end{remark}

Direct computation yields \cite{FS:1987}

\begin{lemma}
Let $T$ be the Virasoro field in the coordinate $x$.
We consider a coordinate change $x\mapsto \hat{x}(x)$ such that $\hat{x}\in\mathcal{D}(S)$, 
and set
\begin{align}\label{coordinate transformation rule for T}
\hat{T}(\hat{x})\left[\frac{d\hat{x}}{dx}\right]^2=T-\frac{c}{12}S(\hat{x}).1\:.
\end{align}
Then $\hat{T}$ satisfies the OPE (\ref{OPE}) in $\hat{x}$. 
\hfill$\Box$
\end{lemma}


\begin{cor}\label{Corollary on two-form} 
Let $X_g$ be a Riemann surface of genus $g\geq 2$ with a complex projective coordinate covering
(i.e. a covering by coordinate patches whose respective local coordinates differ by a M\" obius transformation only).
Then for any state $\langle\quad\rangle$ on $X$, and for any local coordinate $x$ in this class, 
$\langle T(x)\rangle\:(dx)^2$ defines a global section of $(T^*X)^{\otimes 2}$.
\end{cor}

This section is \textit{holomorphic} by assumption.

\begin{proof}\textbf{(e.g.\ \cite{Zhu:1994})}
By property \ref{S(lin.fract.)} of the Schwarzian derivative, and by eq.\ (\ref{coordinate transformation rule for T}), 
\begin{displaymath}
\langle T(x)\rangle\:(dx)^2
=\langle\hat{T}(\hat{x})\rangle\:(d\hat{x})^2\:.\qedhere
\end{displaymath}
\end{proof}

Any compact Riemann surface $X_g$ of $g\geq 2$ admits a projective structure \cite{Gu:1966}. 
Moreover, there is a one-to-one correspondence between projective structures on $X$ and projective connections on $X$.
For any choice of a local coordinate $z$ and of a projective connection $\1^{-1}\mathcal{C}(z)$ on $X$, 
\begin{align}\label{difference of projective connections defines a global section}
\langle T(z)\rangle\:(dz)^2-\frac{c}{12}\:\mathcal{C}(z)\:\in\:H^0((T^*X)^{\otimes 2})\:,
\end{align}
i.e., the difference connection defines a global section in $(T^*X)^{\otimes 2}$ \cite{Gu:1966}.
By the Riemann-Roch Theorem (e.g.\ \cite{FK:1980}),
the affine linear space of projective connections has dimension
$$\cdim\:H^0((T^*X)^{\otimes 2})=3(g-1).$$


\begin{example}
Let $X_g$ be a Riemann surface of arbitrary genus. 
Let $T$ be defined by holomorphic fields of massless free fermions on $X_g$.
In this case, the projective connection $\1^{-1}\mathcal{C}$ is known as the Bergman projective connection 
(\cite{Haw-Schiff:1966},\cite{Fay:1973},\cite{R:1995}). 
\end{example}

\begin{example}
Let $g=1$.
Then $T^*X$ is trivial. When one uses local coordinates given by the affine structure on $X$ \cite{Gu:1966},
then $\langle T(z)\rangle$ is constant.
\end{example}

\section{The Virasoro $1$-point function}

Associate to the hyperelliptic surface $X$ its field of meromorphic functions $K=\C(x,y)/\langle y^2-p(x)\rangle$.
Then $K$ is a field extension of $\C$ of trancendence degree one,
and the two sheets (corresponding to the two signs of $y$) are exchanged by a Galois transformation.

In what follows, we set
\begin{align*}
p(x)=\sum_{k=0}^na_kx^{n-k}, 
\end{align*}
where $n=2g+1$, or $n=2g+2$. 
For convenience of application, we shall treat both cases separately throughout this section, though they are of course equivalent.

\begin{theorem}\label{theorem one-point function}
(\textbf{On the Virasoro one-point function})

For $g\geq 1$, let $X_g$ be the genus $g$ hyperelliptic Riemann surface
\begin{align*}
X:\quad y^2=p(x),
\end{align*}
where $p$ is a polynomial with $\deg p=n$.
\begin{enumerate}
\item\label{behaviour of one-point function in infinity}
As $x\rechts\infty$, 
\begin{align*}
\langle T(x)\rangle\sim\:& x^{-4}\:,\quad\quad\quad\quad\quad\quad\quad\quad\quad\text{for even}\:\:n\:,\\
\langle T(x)\rangle\=\frac{c}{32}\:x^{-2}\1+O(x^{-3})\:,\quad\quad\text{for odd}\:\:n\:.
\end{align*}
\item
We have
\begin{align}\label{p<T(x)>}
p\langle T(x)\rangle
=\frac{c}{32}\frac{[p']^2}{p}\1+\frac{1}{4}\Theta(x,y)\:,
\end{align}
where $\Theta(x,y)$ is a \textbf{polynomial} in $x$ and $y$. More specifically,
we have the Galois splitting
\begin{align}\label{Theta(even) and Theta(odd)}
\Theta(x,y)=\Theta^{[1]}(x)+y\Theta^{[y]}(x). 
\end{align}
Here $\Theta^{[1]}$ is a polynomial in $x$ of degree $n-2$ with the following property:
\begin{enumerate}
\item If $n$ is even, $\quad\left[\Theta^{[1]}+\frac{c}{8}\frac{[p']^2}{p}\1\right]_{>n-4}=0$.
\item\label{odd n}
If $n$ is odd, $\quad\left[\Theta^{[1]}+\frac{c}{8}\left(n^2-1\right)a_0x^{n-2}\1\right]_{>n-3}=0$.
\end{enumerate}
$\Theta^{[y]}$ is a polynomial in $x$ of degree $\frac{n}{2}-4$ if $n$ is even, 
resp.\ $\frac{n-1}{2}-3$ if $n$ is odd, provided $g\geq 3$.
\item\label{dim of space of Virasoro one-point function} 
Let $g\geq 2$. Then the space of $\langle T(x)\rangle$ has dimension $3(g-1)$. 
\end{enumerate}
\end{theorem}

\begin{remark}
The number of degrees of freedom 
in Theorem \ref{theorem one-point function}.\ref{dim of space of Virasoro one-point function} for $g\geq 2$
equals the dimension of the automorphism group of the Riemann surface $X_g$, which in genus $g=0$ and $g=1$
is $\cdim SL(2,\C)=3$ and $\cdim(\C,+)=1$, respectively. 
\end{remark}

\begin{proof}
\begin{enumerate}
\item 
For $x\rechts\infty$, we perform the coordinate change $x\mapsto\tx(x):=\frac{1}{x}$. 
By property \ref{S(lin.fract.)} of the Schwarzian derivative, $S(\tx)=0$ identically, and 
\begin{align*}
T(x)=\tilde{T}(\tx)\left[\frac{d\tx}{dx}\right]^2,
\end{align*}
where $\left[\frac{d\tx}{dx}\right]^2=x^{-4}$.
If $n$ is even, then $\tx$ is an admissible coordinate, so $\langle\tilde{T}(\tx)\rangle$ is holomorphic in $\tx$.
If $n$ is odd, then we may take $\ty:=\sqrt{\tx}$ as coordinate. 
$\frac{d\ty}{dx}=-\frac{1}{2}x^{-1.5}$, and according to eq.\ (\ref{coordinate transformation rule for T}) and eq.\ (\ref{S(p)}),
\begin{align}\label{behaviour of T in infinity if n is odd}
T(x)
=\frac{c}{32}\:x^{-2}+\frac{1}{4}\:\breve{T}(\ty)\:x^{-3},
\end{align}
where $\langle\breve{T}(\ty)\rangle$ is holomorphic in $\ty$. 
\item
$\langle T(x)\rangle$ is a meromorphic function of $x$ and $y$ over $\C$, whence rational in either coordinate.
The ring $\C[x,y]$ of polynomials in $x$ and $y$ is a vector space over the field of rational functions in $x$, 
spanned by $1$ and $y$. Thus we have a splitting
\begin{align*}
\langle T(x)\rangle
=\langle T(x)\rangle^{[1]}+y\:\langle T(x)\rangle^{[y]}.
\end{align*}
$\langle T(x)\rangle$ is $O(1)$ in $x$ 
iff this holds for its Galois-even and its Galois-odd part individually,
as there can't be cancellations between these. 
We obtain a Galois splitting for $\langle\hat{T}(y)\rangle$ by applying a rational transformation to $\langle T(x)\rangle$.
From (\ref{coordinate transformation rule for T}) and (\ref{S(p)}) follows 
\begin{align*}
p\langle T(x)\rangle^{[1]}
=&\frac{c}{32}\1\frac{[p'(x)]^2}{p(x)}+\frac{1}{4}\:\Theta^{[1]}(x),\\
p\langle T(x)\rangle^{[y]}
=&\frac{1}{4}\:\Theta^{[y]}(x),
\end{align*}
where $\Theta^{[1]}$ and $\Theta^{[y]}$ are rational functions of $x$.
We have
\begin{align*}
\frac{1}{4}&\Theta^{[1]}\\
&=
p\langle T(x)\rangle^{[1]}-\frac{c}{32}\1\frac{[p']^2}{p}
=\frac{1}{4}[p']^2\langle\hat{T}(y)\rangle^{[1]}+\frac{c}{12}\1pS(p).
\end{align*}
The l.h.s.\ is $O(1)$ in $x$ for finite $x$ and away from $p=0$ (so wherever $x$ is an admissible coordinate)
while the r.h.s.\ is holomorphic in $y(x)$ for finite $x$ and away from $p'=0$ (so wherever $y$ is an admissible coordinate).
The r.h.s.\ does not actually depend on $y$ but is a function of $x$ alone. 
Since the loci $p=0$ and $p'=0$ do nowhere coincide, we conclude that $\Theta^{[1]}$ is an \textit{entire} function on $\C$.
It remains to check that $\Theta^{[1]}$ has a pole of the correct order at $x=\infty$.
We have
\begin{align}\label{p' squared over p}
\frac{[p']^2}{p}
=n^2a_0x^{n-2}+n(n-2)a_1x^{n-3}+O(x^{n-4}).
\end{align}
\begin{enumerate}
\item
If $n$ is even, 
then $p\langle T(x)\rangle^{[1]}=O(x^{n-4})$ as $x\rechts\infty$, 
by part \ref{behaviour of one-point function in infinity}. 
By eqs (\ref{p<T(x)>}) and (\ref{p' squared over p}), 
$\Theta^{[1]}(x)$ has degree $n-2$ in $x$. Moreover,
\begin{align*}
\Theta^{[1]}&(x)\nn\\
=&-\frac{c}{8}\left(n^2a_0x^{n-2}+n(n-2)a_1x^{n-3}\right)\1+O(x^{n-4}).
\end{align*}
\item
If $n$ is odd, 
then $p \langle T(x)\rangle^{[1]}=\frac{c}{32}a_0x^{n-2}\1+O(x^{n-3})$ as $x\rechts\infty$, 
by eq.\ (\ref{behaviour of T in infinity if n is odd}).
Thus $\Theta^{[1]}$ has degree $n-2$ in $x$.
Moreover, by eq.\ (\ref{p<T(x)>}) and eq.\ (\ref{p' squared over p}),
\begin{align*}
\Theta^{[1]}(x)
&=-\frac{c}{8}\left(n^2-1\right)a_0x^{n-2}\1+O(x^{n-3}). 
\end{align*}
\end{enumerate}
Likewise, we have
\begin{align*}
\frac{1}{4}y\Theta^{[y]}(x)
=yp\langle T(x)\rangle^{[y]}
=\frac{1}{4}[p']^2y\langle\hat{T}(y)\rangle^{[y]};
\end{align*}
the l.h.s.\ is $O(1)$ in $x$ wherever $x$ is an admissible coordinate
while the r.h.s.\ is holomorphic in $y$ wherever $y$ is an admissible coordinate.
Since $y$ is a holomorphic function in $x$ and in $y$ away from $p=0$ and away from $p'=0$, respectively,
this is also true for 
\begin{align*}
\frac{1}{4}&p\Theta^{[y]}(x)
=p^2\langle T(x)\rangle^{[y]}
=\frac{1}{4}p[p']^2\langle\hat{T}(y)\rangle^{[y]}.
\end{align*}
Now the r.h.s.\ does no more depend on $y$ but is a function of $x$ alone, 
so the above argument applies to show that $p\Theta^{[y]}=:P$ is an entire function and thus a polynomial in $x$.
We have $p|P$:
\begin{displaymath}
\frac{P}{y}
=y\Theta^{[y]}(x)=y[p']^2\langle\hat{T}(y)\rangle^{[y]}
\end{displaymath}
is holomorphic in $y$ about $p=0$.
Since $P$ is a polynomial in $x$, and $p$ has no multiple zeros, 
we must actually have $y^2=p$ divides $P$. 
This proves that $\Theta^{[y]}$ is a polynomial in $x$. 
The statement about the degree follows from part \ref{behaviour of one-point function in infinity}.
\item 
This is a consequence of the Riemann-Roch Theorem, cf.\ Section \ref{section: The Virasoro OPE}. 
\end{enumerate}
\end{proof}

\begin{remark}
The main purpose of Theorem \ref{theorem one-point function} is to introduce the polynomial $\Theta$.
Part of the results actually follow from classical formulas for the projective connection.  
For instance, for $n$ even and $g\geq 3$, we have \cite{FK:1980}
\begin{align*}
p\langle T(x)\rangle\:(dx)^2
=\frac{c}{12}p\:\mathcal{C}(x)
+\1\sum_{i=0}^{2g-2}\alpha_ix^i(dx)^2
+y\1\sum_{j=0}^{g-3}\beta_jx^j(dx)^2\:
\end{align*}
in the notations of (\ref{difference of projective connections defines a global section}).
Here the projective connection $\1^{-1}\mathcal{C}$ on $X$ is given by
\begin{align*}
p\:\mathcal{C}(x)
=\frac{3}{8}\left[\frac{[p']^2}{p}\right]_{\leq n-4}\1(dx)^2\:,
\end{align*}
and
\begin{align*}
\left[\Theta^{[1]}(x)\right]_{\leq n-4}
\=4\1\sum_{i=0}^{2g-2}\alpha_ix^i\:,\quad
\Theta^{[y]}(x)
=4\1\sum_{j=0}^{g-3}\beta_jx^j\:.
\end{align*} 
Eq.\ (\ref{behaviour of T in infinity if n is odd}) (for odd $n$) follows from the formula for $\mathcal{C}(x)$ on p.\ 20 in \cite{Fay:1973}.
\end{remark}

\section{The Virasoro $2$-point function}

\subsection{Calculation of the $2$-point function}

For the polynomial $\Theta=\Theta^{[1]}+y\Theta^{[y]}$ defined by eqs\ (\ref{p<T(x)>}) and (\ref{Theta(even) and Theta(odd)}), we set
\begin{align*}
\Theta^{[1]}=\sum_{k=0}^{n-2}A_kx^{n-2-k},\quad\quad
A_k\in\C.\nn
\end{align*}
It will be convenient to replace $\Theta^{[1]}(x)=:-\frac{c}{8}\Pi(x)$ 
for which we introduce even polynomials $\Pi^{[1]}$ and $\Pi^{[x]}$ such that 
\begin{align}
\Pi(x)
=:\Pi^{[1]}(x)+x\Pi^{[x]}(x).\label{splitting of Theta(G-even) into even and odd part}
\end{align}
Likewise, there are even polynomials $p^{[1]}$ and $p^{[x]}$ such that 
\begin{align}\label{splitting of p into even and odd part}
p(x)=p^{[1]}(x)+xp^{[x]}(x). 
\end{align}

\begin{lemma}\label{Lemma expressions for the sum of peven and podd}
For any even polynomial $q$ of $x$, we have
\begin{align*}
q(x_1)+&q(x_2)+O((x_1-x_2)^4)\\
=2q&(\sqrt{x_1x_2})+(x_1-x_2)^2\frac{1}{4}\left(\frac{q'(\sqrt{x_1x_2})}{\sqrt{x_1x_2}}+q''(\sqrt{x_1x_2})\right)\:,
\end{align*}
and
\begin{align*}
x_1q(x_1)+&x_2q(x_2)+O((x_1-x_2)^4)\\
=(x_1+x_2)
&\left\{q(\sqrt{x_1x_2})
+(x_1-x_2)^2\frac{1}{8}
\left(\frac{3q'(\sqrt{x_1x_2})}{\sqrt{x_1x_2}}+q''(\sqrt{x_1x_2})\right)
\right\}\:.
\end{align*}
\end{lemma}

Note that the polynomials $q$ and $q''$ in $\sqrt{x_1x_2}$ are are actually polynomials in $x_1x_2$.

\begin{proof} Direct computation. The calculation can be shortened by using
\begin{align*}
x_1&=(1+\eps)\:x,\\
x_2&=(1-\eps)\:x,
\end{align*}
where $|\eps|\ll 1$.
\end{proof}

Abusing notations, for $j=1,2$, we shall write $p_j=p(x_j)$ and $\Theta_j=\Theta(x_j,y_j)$
etc. 
For $k\geq 0$, we denote by $[R(x_1,x_2)]_{>k}$ the polynomial in $x=x_1$ defined by eq.\ (\ref{projection of a rational function onto a part of bounded degree}),
with $x_2$ held fixed, and let $[R(x_1,x_2)]^{>k}$ be the polynomial for the opposite choice $x=x_2$ ($x_1$ fixed).  

\begin{theorem}\label{Theorem: Virasoro two-point function}
(\textbf{The Virasoro two-point function})

For $g\geq 1$, let $X_g$ be the hyperelliptic Riemann surface
\begin{align*}
X:\quad y^2=p(x),
\end{align*}
where $p$ is a polynomial, $\deg p=n$ odd.
Let 
\begin{displaymath}
\langle T(x_1)T(x_2)\rangle_c
:=\1^{-1}\langle T(x_1)T(x_2)\rangle-\1^{-2}\langle T(x_1)\rangle\langle T(x_2)\rangle
\end{displaymath}
be the connected two-point function of the Virasoro field. 
We have
\begin{enumerate}
\item 
\begin{align}\label{behaviour of Virasoro two-point function as x goes to infty}
\langle T(x_1)T(x_2)\rangle_c\:p_1p_2=O(x_1^{n-3}).
\end{align}
\item
For $|x_1|,|x_2|$ small, 
\begin{align*}
\langle T(x_1)T(x_2)\rangle_c\:p_1p_2
=\1^{-1}R(x_1,x_2)+O\left(1\right)|_{x_1=x_2}\:,
\end{align*}
where $R(x_1,x_2)$ is a rational function of $x_1,x_2$ and $y_1,y_2$,
and $O(1)|_{x_1=x_2}$ denotes terms that are regular on the diagonal $x_1=x_2$. 
\item\label{rational function of the two-point function}
The rational function is given by
\begin{align*}
R(x_1,x_2)
&=\frac{c}{4}\1\frac{p_1p_2}{(x_1-x_2)^4}\nn\\
&+\frac{c}{4}y_1y_2\1\left(\frac{p^{[1]}(\sqrt{x_1x_2})}{(x_1-x_2)^4} 
+\frac{1}{2}(x_1+x_2)\frac{p^{[x]}(\sqrt{x_1x_2})}{(x_1-x_2)^4}\right)\nn\\
&+\frac{c}{32}\1\frac{p_1'p_2'}{(x_1-x_2)^2}\\
&+\frac{c}{32}y_1y_2\1\left(\frac{\frac{1}{\sqrt{x_1x_2}}\:p^{[1]'}(\sqrt{x_1x_2})}{(x_1-x_2)^2}
+\frac{3}{2}(x_1+x_2)\frac{\frac{1}{\sqrt{x_1x_2}}\:p^{[x]'}(\sqrt{x_1x_2})}{(x_1-x_2)^2}\right)\nn\\
&+\frac{1}{8}\frac{p_1\Theta_2+p_2\Theta_1}{(x_1-x_2)^2}\\
&+\frac{1}{8}\left(y_1\Theta_2^{[y]}+y_2\Theta_1^{[y]}\right)
\left(\frac{p^{[1]}(\sqrt{x_1x_2})}{(x_1-x_2)^2}
+\frac{1}{2}(x_1+x_2)\frac{p^{[x]}(\sqrt{x_1x_2})}{(x_1-x_2)^2}\right)\nn\\
&+\frac{c}{32}y_1y_2\1\left(\frac{p^{[1]''}(\sqrt{x_1x_2})}{(x_1-x_2)^2}
+\frac{1}{2}(x_1+x_2)\frac{p^{[x]''}(\sqrt{x_1x_2})}{(x_1-x_2)^2}\right)\nn\\
&-\frac{c}{32}y_1y_2\left(\frac{\Pi^{[1]}(\sqrt{x_1x_2})}{(x_1-x_2)^2} 
+\frac{1}{2}(x_1+x_2)\frac{\Pi^{[x]}(\sqrt{x_1x_2})}{(x_1-x_2)^2}\right).\nn
\end{align*}
Here $p^{[1]}$ and $p^{[x]}$ and $\Pi^{[1]}$ and $\Pi^{[x]}$ are the even polynomials 
introduced in (\ref{splitting of p into even and odd part}) and in (\ref{splitting of Theta(G-even) into even and odd part}), respectively.
\item\label{O(1) part of connected Virasoro two-point function} 
For $R(x_1,x_2)$ thus defined, the connected Virasoro two-point function reads
\begin{align*}
\1\langle T(x_1)T(x_2)\rangle_c&\:p_1p_2\\
&=R(x_1,x_2)+P(x_1,x_2,y_1,y_2)\\
 &-\frac{1}{8}a_0\left(x_1^{n-2}\Theta_2+x_2^{n-2}\Theta_1\right)
 -\frac{c}{64}\1(n^2-1)a_0^2x_1^{n-2}x_2^{n-2}\\
 &-\frac{1}{8}y_1a_1x_1^{\frac{n}{2}-\frac{5}{2}}x_2^{\frac{n}{2}-\frac{1}{2}}\Theta_2^{[y]}
 -\frac{1}{8}y_2a_1x_1^{\frac{n}{2}-\frac{1}{2}}x_2^{\frac{n}{2}-\frac{5}{2}}\Theta_1^{[y]}\\
&-\frac{1}{16}y_1a_0x_1^{\frac{n}{2}-\frac{3}{2}}x_2^{\frac{n}{2}-\frac{1}{2}}\Theta_2^{[y]}
-\frac{1}{16}y_2a_0x_1^{\frac{n}{2}-\frac{1}{2}}x_2^{\frac{n}{2}-\frac{3}{2}}\Theta_1^{[y]}\\ 
 &-\frac{3}{16}y_1a_0x_1^{\frac{n}{2}-\frac{5}{2}}x_2^{\frac{n}{2}+\frac{1}{2}}\Theta_2^{[y]}
 -\frac{3}{16}y_2a_0x_1^{\frac{n}{2}+\frac{1}{2}}x_2^{\frac{n}{2}-\frac{5}{2}}\Theta_1^{[y]}\\ 
 &-\frac{1}{16}y_1a_2x_1^{\frac{n}{2}-\frac{5}{2}}x_2^{\frac{n}{2}-\frac{3}{2}}\Theta_2^{[y]}
 -\frac{1}{16}y_2a_2x_1^{\frac{n}{2}-\frac{3}{2}}x_2^{\frac{n}{2}-\frac{5}{2}}\Theta_1^{[y]}\:,
\end{align*}
where 
\begin{align}\label{definition of the polynomial which is specific to the state}
P(x_1,&x_2,y_1,y_2)\\
\=P^{[1]}(x_1,x_2)+y_1P^{[y_1]}(x_1,x_2)+y_2P^{[y_2]}(x_1,x_2)+y_1y_2P^{[y_1y_2]}(x_1,x_2)\:.\nn 
\end{align}
Here $P^{[1]}$, $P^{[y_1y_2]}$ and for $i=1,2$, $P^{[y_i]}$ are polynomials in $x_1$ and $x_2$ with
\begin{align*}
\deg_iP^{[1]}
\=n-3
=\deg_iP^{[y_j]}\:,\quad\text{for}\:j\not=i\:,\\
\deg_iP^{[y_i]}
\=\frac{n-1}{2}-3
=\deg_iP^{[y_1y_2]}
\:.
\end{align*}
($\deg_i$ denotes the degree in $x_i$).
Moreover, $P^{[1]}$, $P^{[y_1y_2]}$ and $y_1P^{[y_1]}+y_2P^{[y_2]}$ are symmetric
under flipping $1\leftrightarrow 2$.
These four polynomials are specific to the state. 
\end{enumerate}
\end{theorem}

\begin{proof}
Direct computation (cf.\ Appendix).
\end{proof}

In the following, let $\left[T(x_1)T(x_2)\right]_{\reg}+\langle T(x_1)\rangle_c\langle T(x_2)\rangle_c.1$ 
with $\langle T(x)\rangle_c=\1^{-1}\langle T(x)\rangle$
be the regular part of the OPE (\ref{OPE}) on the hyperelliptic Riemann surface $X$,
\begin{align}
T(x_1)T(x_2)\:p_1p_2
\mapsto
&\:\frac{c}{32}f(x_1,x_2)^2.1
+\frac{1}{4}f(x_1,x_2)\left(\vartheta_1+\vartheta_2\right)\nn\\
&\:+\left[T(x_1)T(x_2)\right]_{\reg}\:p_1p_2+\langle T(x_1)\rangle_c\langle T(x_2)\rangle_cp_1p_2.1
\label{definition of [T1T2]reg}\:.
\end{align}
Here $f(x_1,x_2):=\left(\frac{y_1-y_2}{x_1-x_2}\right)^2$,
and
\begin{align}\label{definition of vartheta}
\vartheta(x)
:=T(x)\:p-\frac{c}{32}\frac{[p']^2}{p}.1
\end{align}
satisfies $\langle\vartheta(x)\rangle=\frac{1}{4}\:\Theta(x,y)$.
$\vartheta(x)$ is holomorphic about $p=0$ since by eqs (\ref{S(p)}) and (\ref{coordinate transformation rule for T}), 
\begin{displaymath}
\hat{T}(y)
=\frac{4}{[p']^2}\left(\vartheta-\frac{c}{12}S(p)\:p\right)\:,
\end{displaymath}
where $S(p)$ is regular at $p=0$. 

\subsection{Application to the $(2,5)$ minimal model, in the case $n=5$}

In Section \ref{section: The Virasoro OPE} we introduced the so-called normal ordered product
\begin{align*}
N_0(\varphi_1,\varphi_2)(x_2)
=\lim_{x_1\rechts x_2}\left[\varphi_1(x_1)\varphi_2(x_2)\right]_{\text{regular}}
\end{align*}
of two fields $\varphi_1$, $\varphi_2$, 
where $\left[\varphi_1(x_1),\varphi_2(x_2)\right]_{\text{regular}}$ is the regular part of the OPE of $\varphi_1$, $\varphi_2$. 
In particular, $\langle N_0(T,T)(x)\rangle$ can be determined from 
Theorem \ref{Theorem: Virasoro two-point function}.\ref{O(1) part of connected Virasoro two-point function}.
To illustrate our formalism, we provide a short proof of the following well-known result (\cite{Blum:2009}, Sect.\ 3):

\begin{lemma}
The condition
$N_0(T,T)\propto\partial^2 T$ implies $c=-\frac{22}{5}$ and
\begin{align}
N_0(T,T)(x)
=\frac{3}{10}\:\partial^2 T(x).\label{(2,5) minimal model}
\end{align}
\end{lemma}

\begin{proof}
The statement is local, so we may assume w.l.o.g.\ $g=1$. In this case, 
\begin{align*}
\Theta^{[1]}(x)
&=-4cx\1+A_1,
\quad\quad
\Theta^{[y]}=0,
\end{align*}
by Theorem \ref{theorem one-point function}.(\ref{odd n}).
Using Corollary \ref{Corollary on two-form} and the transformation rule (\ref{coordinate transformation rule for T}),
we find
\begin{align*}
\langle T(x)\rangle
&=\frac{c}{32}\frac{[p']^2}{p^2}\1-c\1\frac{x}{p}+\frac{\T}{p},
\end{align*} 
where by (\ref{p<T(x)>}), $\T=\frac{A_1}{4}$.
Direct computation shows that 
\begin{displaymath}
\langle N_0(T,T)(x)\rangle=\alpha\:\partial^2\langle T(x)\rangle 
\end{displaymath}
iff $\alpha=\frac{3}{10}$ and $c=-\frac{22}{5}$. Since by assumption the two underlying fields are proportional, the claim follows.
\end{proof}

The aim of this section is to determine at least some of the constants in the Virasoro two-point function 
in the $(2,5)$ minimal model for $g=2$.
We will restrict our considerations to the case when $n$ is \textit{odd}.
(Better knowledge about $\Theta^{[1]}$ when $n$ is even doesn't actually provide more information,
it just leads to longer equations.) In the first case to consider, namely $n=5$, all Galois-odd terms are absent. 
Restricting to the Galois-even terms, condition (\ref{(2,5) minimal model}) reads as follows:

\begin{lemma}
In the $(2,5)$ minimal model for $g\geq 1$, we have
 \begin{align*}
&\quad\frac{7c}{640}\1\frac{[p'']^2}{p^2}
-\frac{7c}{960}\1\frac{p'p'''}{p^2}
+\frac{c}{1536}\frac{p^{(4)}}{p}\\
&+\frac{1}{20}\frac{p''}{p^2}\Theta^{[1]}
+\frac{3}{80}\frac{p'}{p^2}\Theta^{[1]'}
-\frac{3}{160}\frac{\Theta^{[1]''}}{p}\\
&-\frac{1}{16}\1^{-1}\left(\frac{\Theta^{[1]^2}}{p^2}+\frac{\Theta^{[y]^2}}{p}\right)\\
&+\frac{1}{4}a_0\frac{x^{n-2}}{p^2}\Theta^{[1]}
-\frac{1}{8}A_0a_0\frac{x^{2n-4}}{p^2}\\
&-\frac{c}{8\cdot32}\:\frac{1}{xp}\left(\Pi^{[1]'}+x\Pi^{[x]'}\right)\\
&-\frac{c}{256}\frac{1}{xp}\1\left(-p^{(3)}
-\frac{1}{2}\left(\frac{p^{[1]''}}{x}-p^{[x]''}\right)
+\frac{1}{2x}\left(\frac{p^{[1]'}}{x}+5p^{[x]'}\right)\right)\\
=&\quad\frac{P^{[1]}(x,x)}{p^2}+\frac{P^{[y_1y_2]}(x,x)}{p}.
\end{align*}
\end{lemma}
Note that the equation makes good sense since the l.h.s.\ is regular at $x=0$. 
For instance, $\Pi^{[1]'}$ is an odd polynomial of $x$, so its quotient by $x$ is regular. 

\begin{proof}
Direct computation.
\end{proof}

\begin{example}\label{example: (2,5) minimal model for n=5}
When $n=5$,
\begin{displaymath}
\deg P^{[1]}(x,x)=4,\quad P^{[y_1y_2]}(x,x)=0\:.
\end{displaymath}
Thus we have $5$ complex degrees of freedom.
One of them is the number $\1$, 
and according to Theorem \ref{theorem one-point function}.\ref{dim of space of Virasoro one-point function},
at most $3$ of them are given by $\langle T(x)\rangle$.
Set
\begin{align*}
P^{[1]}(x_1,x_2)
&=
B_{2,2}\:x_1^2x_2^2\\
&+B_{2,1}(x_1^2x_2+x_1x_2^2)\\
&+B_2(x_1^2+x_2^2)
+B_{1,1}\:x_1x_2\\
&+B_1(x_1+x_2)\\
&+B_0\:,
\end{align*}
$B_0,B_1,B_{i,j}\in\1\C$, $i,j=1,2$.
The additional constraint (\ref{(2,5) minimal model}) provides knowledge of
\begin{align*}
P^{[1]}(x,x)
=&
B_{2,2}\:x^4
+2B_{2,1}\:x^3
+(2B_2+B_{1,1})\:x^2
+2B_1\:x
+B_0
\end{align*}
only, so we are left with one unknown.
We will see later that all constants can be fixed using (\ref{(2,5) minimal model})
when the three-point function is taken into account.
\end{example}

\section{The Virasoro $N$-point function}

\subsection{Graph representation of the Virasoro $N$-point function}

For $g\geq 1$, let $X_g$ be the genus $g$ hyperelliptic Riemann surface
\begin{align*}
X:\quad y^2=p(x),
\end{align*}
where $p$ is a polynomial, $\deg p=n$, with $n=2g+1$, or $n=2g+2$.
Let $\mathcal{F}$ be the bundle of holomorphic fields introduced in Section \ref{section: The Virasoro OPE}.

\begin{theorem}
Let $S(x_1,\ldots,x_N)$, $N\in\N$, be the set of oriented graphs with vertices $x_1,\ldots,x_N$ (not necessarily connected), 
subject to the condition that every vertex has at most one ingoing and at most one outgoing line.

Given a state $\langle\quad\rangle$ on $X$, 
there is a multilinear map
\begin{align*}
\langle\quad\rangle_r:
\quad
S_*(\mathcal{F})\rechts\C\:,
\end{align*}
normalised such that $\1_r=\1$, with the following properties:
\begin{enumerate}
\item
For all $k\in\N$, $k\geq 2$, and any $\varphi_2,\ldots,\varphi_k\in\{1,T\}$, 
we have
\begin{displaymath}
\langle 1\:\varphi_2(z_2)\ldots\varphi_k(z_k)\rangle_r=\langle\varphi_2(z_2)\ldots\varphi_k(z_k)\rangle_r\:. 
\end{displaymath}
\item\label{the r-k-point function of the vartheta field}
For all $k\in\N$, 
$\langle\vartheta_1\ldots\vartheta_k\rangle_r$
is a polynomial in $x_1,\ldots,x_k$ and $y_1,\ldots,y_k$. 
\item
We have
\begin{align}\label{graph rep}
\langle T(x_1)\ldots T(x_N)\rangle\:p_1\ldots p_N
=\sum_{\Gamma\in S(x_1,\ldots,x_N)}F(\Gamma)\:,
\end{align}
where for $\Gamma\in S(x_1,\ldots,x_N)$,
\begin{align*}
F(\Gamma)
:=\left(\frac{c}{2}\right)^{\sharp\text{loops}}
\prod_{(x_i,x_j)\in\Gamma}\left(\frac{1}{4}\:f(x_i,x_j)\right)
\left\langle\bigotimes_{k\in A_N\cap {E_N}^c}\vartheta(x_k)\bigotimes_{\ell\in(A_N\cup E_N)^c}T(x_{\ell})p_{\ell}\right\rangle_r\:. 
\end{align*}
Here for any oriented edge $(x_i,x_j)\in\Gamma$, 
\begin{displaymath}
f(x_i,x_j)
:=\left(\frac{y_i+y_j}{x_i-x_j}\right)^2\:, 
\end{displaymath}
and $\vartheta$ is the field defined by eq.\ (\ref{definition of vartheta}).
$A_N, E_N\subset\{1,\ldots,N\}$ are the subsets
\begin{align*}
A_N:=&\{i\:|\exists\:j\:\text{such that}\:(x_i,x_j)\in\Gamma\}\:,\\
E_N:=&\{j\:|\exists\:i\:\text{such that}\:(x_i,x_j)\in\Gamma\}\:. 
\end{align*}
$\cup$ and $\cap$ are the set theoretic union and intersection, respectively, and $(\ldots)^c$ denotes the complement in $\{1,\ldots,N\}$.
\end{enumerate}
\end{theorem}

Note that $\langle\quad\rangle_r$ is not a state (it is not compatible with the OPE).

\begin{proof}
We use induction on $N$. 
By multilinearity of $\langle\quad\rangle_r$ and eq.\ (\ref{definition of vartheta}), 
$F(\Gamma)$ for $\Gamma\in S(x_1,\ldots,x_N)$ is determined by 
$\langle T(x_1)\ldots T(x_k)\rangle_r$, for $k\leq N$.

Suppose $\langle T(x_1)\ldots T(x_k)\rangle_r$, for $k\leq N$ has the required properties for $k<N$.
We define 
\begin{displaymath}
\langle T(x_1)\ldots T(x_N)\rangle_r
\end{displaymath}
by (\ref{graph rep}) and show first that 
$\langle T(x_1)\ldots T(x_N)\rangle_r$ is regular as two positions coincide.
In other words,
let $\Gamma_0(x_1,\ldots,x_N)\in S(x_1,\ldots,x_N)$ be the graph whose vertices are all isolated.
Then $\sum_{\Gamma\not=\Gamma_0}F(\Gamma)$ reproduces the correct singular part of the Virasoro $N$-point function
as prescribed by the OPE (\ref{definition of [T1T2]reg}) on $X$.

For $N=1$, $\Gamma_0(x)$ is the only graph, and 
\begin{align}\label{Virasoro one-point function r}
\langle T(x)\rangle\:p
=F(\Gamma_0(x))
=\langle T(x)\rangle_r p\:.
\end{align}

For  $N=2$, the admissible graphs form a closed loop,
a single line segment (with two possible orientations), and two isolated points.
Thus by eq.\ (\ref{graph rep}),
\begin{align*}
\langle T(x_1)T(x_2)\rangle_rp_1p_2
\=\langle T(x_1)T(x_2)\rangle\:p_1p_2-\frac{c}{32}f_{12}^2\1_r-\frac{1}{4}f_{12}\langle\vartheta_1+\vartheta_2\rangle_r
\:, 
\end{align*}
where $f_{12}:=f(x_1,x_2)$. 
According to the OPE (\ref{definition of [T1T2]reg}),
$\langle T(x_1)T(x_2)\rangle_rp_1p_2$ is regular on the diagonal $x_1=x_2$.

In order to prove regularity of $\langle T(x_1)\ldots T(x_N)\rangle_rp_1\ldots p_N$ on all partial diagonals for $N>2$,
it suffices to show that the coefficients of the singularities are correct.
Suppose the graph representation for the $k$-point function of the Virasoro field is correct for $2\leq k\leq N-1$.
For $1\leq i\leq N$, set 
$S^{[i]}:=S(x_i,\ldots,x_N)$
and
$\Gamma_0^{[i]}:=\Gamma_0(x_i,\ldots,x_N)$.
For $1\leq i,j\leq N$, $i\not=j$, define
\begin{align*}
S_{(ij)}
:=&\{\Gamma\in S^{[1]}|\:(i,j),(j,i)\in\Gamma\}\:,\\
S_{(i,j)}
:=&\{\Gamma\in S^{[1]}|\:(i,j)\in\Gamma,(j,i)\notin\Gamma\}\:,\\
S_{(i)(j)}
:=&\{\Gamma\in S^{[1]}|\:(i,j),(j,i)\notin\Gamma\}\:.
\end{align*}
$S^{[1]}$ decomposes as
\begin{displaymath}
S^{[1]}
=S_{(12)}\cup S_{(1,2)}\cup S_{(2,1)}\cup S_{(1),(2)}. 
\end{displaymath}
Since $S_{(12)}\cong S^{[3]}$, the equality
\begin{align*}
\sum_{\Gamma\in S_{(12)}}F(\Gamma) 
=\frac{c}{32}\:f_{12}^2\:\langle T(x_3)\ldots T(x_N)\rangle\:p_3\ldots p_N
\end{align*}
(with $f_{12}:=f(x_1,x_2)$) holds by the induction hypothesis. 
By the symmetrization argument following eq.\ (\ref{expansion of second order term}),
it remains to show that
\begin{displaymath}
\sum_{\Gamma\in S^{[1]}\setminus S_{(12)}}F(\Gamma)
=\frac{f_{12}}{2}\:\langle\vartheta_2T(x_3)\ldots T(x_N)\rangle\:p_3\ldots p_N+O((x_1-x_2)^{-1})\:,
\end{displaymath} 
which under the induction hypothesis on $S^{[2]}$ and $S^{[3]}$, we reformulate as
\begin{align*}
\sum_{\Gamma\in S^{[2]}}&F(\varphi^{-1}(\Gamma))+F(\bar{\varphi}^{-1}(\Gamma))\\ 
=&\:\frac{f_{12}}{2}\left(\sum_{\Gamma\in S^{[2]}}F(\Gamma)
-\frac{c}{32}\frac{[p_2']^2}{p_2}\sum_{\Gamma'\in S^{[3]}}F(\Gamma')\right)+O((x_1-x_2)^{-1})\:.
\end{align*}
Here $\varphi$, $\bar{\varphi}$ are the invertible maps
\begin{align*}
\varphi:\quad S_{(1,2)}\rechts S^{[2]},\\
\bar{\varphi}:\quad S_{(2,1)}\rechts S^{[2]},                               
\end{align*}
given by contracting the link $(x_1,x_2)$ resp.\ $(x_2,x_1)$ into the point $x_2$, and leaving the graph unchanged otherwise.
Let $S_{(2)}\subset S^{[2]}$ be the subset of graphs containing $x_2$ as an isolated point,
and let $\chi: S_{(2)}\rechts S^{[3]}$ be the isomorphism given by omitting the vertex $x_2$ from the graph.
Now for $\Gamma\in {S_{(2)}}$, the graph representation yields
\begin{align*}
F(&\varphi^{-1}(\Gamma))+F(\bar{\varphi}^{-1}(\Gamma))\\ 
&=\:\frac{f_{12}}{2}\:\left(F(\Gamma)-\frac{c}{32}\frac{[p_2']^2}{p_2}\:F(\chi(\Gamma))\right)
+O((x_1-x_2)^{-1})\:,
\end{align*}
while for $\Gamma\in S^{[2]}\setminus S_{(2)}$, 
\begin{displaymath}
F(\varphi^{-1}(\Gamma))+F(\bar{\varphi}^{-1}(\Gamma)) 
=\:\frac{f_{12}}{2}\:F(\Gamma)+O((x_1-x_2)^{-1})\:.
\end{displaymath}

It remains to show part \ref{the r-k-point function of the vartheta field}.
It is sufficient to show that $\langle\vartheta_1\ldots\vartheta_k\rangle_r$ for $1\leq k\leq N$ is regular at $p_1=0$.
We use induction. 
From eqs (\ref{definition of vartheta}) and (\ref{Virasoro one-point function r}) follows 
$\langle\vartheta\rangle_r
=\langle\vartheta\rangle
=\frac{1}{4}\Theta$,
which is a polynomial in $x$ and $y$.
Now suppose $\langle\vartheta_1\ldots\vartheta_k\rangle_r$ is regular at $p_1=0$, for $k\leq N-1$.
For $I\in\mathcal{P}(\{2,\ldots,N\})$, the powerset of $\{2,\ldots,N\}$, 
we have by eq.\ (\ref{definition of vartheta}), 
\begin{align*}
\langle\vartheta_1\ldots\vartheta_N\rangle_r 
\=\langle T(x_1)\ldots T(x_N)\rangle_rp_1\ldots p_N
-\frac{c}{32}\frac{[p_1']^2}{p_1}
\sum_{I\in\mathcal{P}(\{2,\ldots,N\})}
\prod_{i\in I}\left(\frac{c}{32}\frac{[p_i']^2}{p_i}\right)
\langle\bigotimes_{j\in I^c}\vartheta_j\rangle_r\\
&\hspace{4cm}+\text{terms regular at}\:p_1=0\\
\=\langle T(x_1)\ldots T(x_N)\rangle_rp_1\ldots p_N
-\frac{c}{32}\frac{[p_1']^2}{p_1}\langle T(x_2)\ldots T(x_N)\rangle_rp_2\ldots p_N\\
&\hspace{4cm}+\text{terms regular at}\:p_1=0
\:. 
\end{align*}
Here we have, using the graph representation,
\begin{align*}
\langle T(x_1)\ldots T(x_N)\rangle_rp_1\ldots p_N
\=\langle T(x_1)\ldots T(x_N)\rangle p_1\ldots p_N
-\sum_{\Gamma\in S^{[1]}\setminus\Gamma_0^{[1]}}F(\Gamma)\\
\=\frac{c}{32}\frac{[p_1']^2}{p_1}\left(\langle T(x_2)\ldots T(x_N)\rangle p_2\ldots p_N
-\sum_{\Gamma\in S^{[2]}\setminus\Gamma_0^{[2]}}F(\Gamma)\right)\\
\+\text{terms regular at}\:p_1=0\\
\=\frac{c}{32}\frac{[p_1']^2}{p_1}
\langle T(x_2)\ldots T(x_N)\rangle_r p_2\ldots p_N\\
\+\text{terms regular at}\:p_1=0
\:.
\end{align*}
We explain the second identity.
Consider the augmentation map $a: S^{[2]}\setminus\Gamma_0^{[2]}\rechts S^{[1]}\setminus\Gamma_0^{[1]}$
defined by adjoining the isolated vertex $x_1$ to the graph. We have
\begin{align*}
F(a(\Gamma))
=\frac{c}{32}\frac{[p_1']^2}{p_1}F(\Gamma)+\{\text{terms regular at}\:p_1=0\}\:. 
\end{align*}
Indeed, all terms in $F(a(\Gamma))$ that involve $\vartheta_1$ are $k$-point functions with $k<N$, 
since end points of edges are not labelled,
and so are regular at $p_1=0$ by assumption.
We conclude that $\langle\vartheta_1\ldots\vartheta_N\rangle_r$ is regular at $p_1=0$.


\end{proof}

Since the proof is by recursion, it should generalise easily to more general Riemann surfaces.  

We illustrate the theorem for the case $N=3$. 
Recall that the connected $1$-point, $2$-point and $3$-point functions are given by 
\begin{align*}
\langle\varphi(x)\rangle_c
\=\1^{-1}\langle\varphi(x)\rangle\:,\\
\langle\varphi_1(x_1)\varphi_2(x_2)\rangle_c
\=\1^{-1}\langle\varphi_1(x_1)\varphi_2(x_2)\rangle
-\1^{-2}\langle\varphi_1(x_1)\rangle\:\langle\varphi_2(x_2)\rangle\:,
\end{align*}
and 
\begin{align*}
\langle\varphi_1(x_1)\varphi_2(x_2)\varphi_3(x_3)\rangle_c
\=\1^{-1}\langle\varphi_1(x_1)\varphi_2(x_2)\varphi_3(x_3)\rangle\\
\-\1^{-2}\left\{\langle\varphi_1(x_1)\varphi_2(x_2)\rangle\:\langle\varphi_3(x_3)\rangle+\cyc\right\}\\
\-\1^{-3}\langle\varphi_1(x_1)\rangle\:\langle\varphi_2(x_2)\rangle\:\langle\varphi_3(x_3)\rangle
\:.
\end{align*}

\begin{example}\label{Example: Virasoro three-point function on X}
When $\deg p=n$ is odd, 
\begin{align*}
\langle T(x_1)T(x_2)T(x_3)\rangle_c\:p_1p_2p_3=O(x_1^{n-3})\:. 
\end{align*}
\item 
In the region where $x_1,x_2,x_3$ are finite,
the connected Virasoro three-point function is given by
\begin{align*}
&\langle T(x_1)T(x_1)T(x_3)\rangle_cp_1p_2p_3
=R^{(0)}(x_1,x_2,x_3)+O(1)|_{x_1,x_2,x_3}\:,
\end{align*}
where $R^{(0)}$ is the rational function of $x_1,x_2,x_3$ and $y_1,y_2,y_3$ given by
\begin{align*}
R^{(0)}(x_1,x_2,x_3)
=&\frac{c}{64}\left(\frac{y_1+y_2}{x_1-x_2}\right)^2\left(\frac{y_1+y_3}{x_1-x_3}\right)^2\left(\frac{y_2+y_3}{x_2-x_3}\right)^2\\
+&\frac{1}{64}
\left(\frac{y_1+y_2}{x_1-x_2}\right)^2
\left(\frac{y_1+y_3}{x_1-x_3}\right)^2\1^{-1}(\Theta_2+\Theta_3)\\
+&\frac{1}{64}
\left(\frac{y_1+y_2}{x_1-x_2}\right)^2
\left(\frac{y_2+y_3}{x_2-x_3}\right)^2\1^{-1}(\Theta_1+\Theta_3)\\
+&\frac{1}{64}
\left(\frac{y_1+y_3}{x_1-x_3}\right)^2
\left(\frac{y_2+y_3}{x_2-x_3}\right)^2\1^{-1}(\Theta_1+\Theta_2)\\
+&\frac{1}{4}\left(\frac{y_1+y_2}{x_1-x_2}\right)^2
\Big(
\1^{-1}\langle\left[T(x_1)T(x_3)\right]_{\reg}\rangle p_1p_3\\
&\quad\quad\quad\quad\quad\quad
+\1^{-1}\langle\left[T(x_2)T(x_3)\right]_{\reg}\rangle p_2p_3\Big)\\
+&\frac{1}{4}\left(\frac{y_1+y_3}{x_1-x_3}\right)^2
\Big(
\1^{-1}\langle\left[T(x_1)T(x_2)\right]_{\reg}\rangle p_1p_2\\
&\quad\quad\quad\quad\quad\quad
+\1^{-1}\langle\left[T(x_2)T(x_3)\right]_{\reg}\rangle\:p_2p_3\Big)\\
+&\frac{1}{4}\left(\frac{y_2+y_3}{x_2-x_3}\right)^2
\Big(
\1^{-1}\langle\left[T(x_1)T(x_2)\right]_{\reg}\rangle\:p_1p_2\\
&\quad\quad\quad\quad\quad\quad
+\1^{-1}\langle\left[T(x_1)T(x_3)\right]_{\reg}\rangle\:p_1p_3\Big).
\end{align*}
Here for $i,j\in\{1,2,3\}$, $\left[T(x_i)T(x_j)\right]_{\reg}$ is defined by (\ref{definition of [T1T2]reg}).
By part \ref{Proof of the Theorem on the Virasoro-2-point function: Comparison with the OPE} 
in the Proof of Theorem \ref{Theorem: Virasoro two-point function},
$\langle\left[T(x_i)T(x_j)\right]_{\reg}\rangle\:p_ip_j$ is a polynomial in $x_i,x_j$ and $y_i,y_j$.

Moreover, the $O(1)|_{x_1,x_2,x_3}$ term is a polynomial in $x_1,x_2,x_3$ and $y_1,y_2,y_3$.
Indeed, $\langle T(x_1)T(x_2)T(x_3)\rangle_c$ is regular at $p_1=0$ because  
\begin{align*}
\langle T(x_1)T(x_2)T(x_3)\rangle_c
\=\1^{-1}\langle T(x_1)T(x_2)T(x_3)\rangle
-\1^{-2}\langle T(x_1)\rangle\:\langle T(x_2)T(x_3)\rangle\\
\-\1^{-1}\left\{\langle T(x_1)T(x_2)\rangle_c\:\langle T(x_3)\rangle
+\langle T(x_3)T(x_1)\rangle_c\:\langle T(x_2)\rangle
\right\}
\:.
\end{align*}
\end{example}

\subsection{Application to the $(2,5)$ minimal model, in the case $n=5$}


\begin{theorem}
We consider the $(2,5)$ minimal model on a genus $g=2$ hyperelliptic Riemann surface
\begin{displaymath}
X:\quad y^2=p(x). 
\end{displaymath}
There are exactly $4$ parameters, given by $\1$ and $\langle T(x)\rangle$,
and all other constants in the two-and three-point function are known.
\end{theorem}

\begin{proof}
W.l.o.g.\ $n=5$. 
In this case the two-point function in the $(2,5)$ minimal model has been determined previously, up to one constant, 
cf.\ Example \ref{example: (2,5) minimal model for n=5}. 
In the three-point function, there is only one polynomial $P^{[1]}(x_1,x_2,x_3)$, of degree $n-3$ in each of $x_1,x_2,x_3$, free to choose. 
Set
\begin{align*}
P^{[1]}(x_1,x_2,x_3)
=&B_{2,2,2}\:x_1^2x_2^2x_3^2\\
+&B_{2,2,1}(x_1^2x_2^2x_3+x_1^2x_2x_3^2+x_1x_2^2x_3^2)\\
+&B_{2,1,1}(x_1^2x_2x_3+x_1x_2^2x_3+x_1x_2x_3^2)\\
&\quad+B_{2,2,0}(x_1^2x_2^2+x_1^2x_3^2+x_2^2x_3^2)\\
+&B_{2,1,0}(x_1^2x_2+x_1^2x_3+x_1x_2^2+x_1x_3^2+x_2^2x_3+x_2x_3^2)\\
&\quad+B_{1,1,1}\:x_1x_2x_3\\
+&B_{2,0,0}(x_1^2+x_2^2+x_3^2)
+B_{1,1,0}(x_1x_2+x_1x_3+x_2x_3)\\
+&B_{1,0,0}(x_1+x_2+x_3)\\
+&B_{0,0,0},
\quad\quad\quad\quad\quad
B_{i,j,k}\in\1\C,\:i,j,k\in\{1,2\},\:k\leq j\leq i.
\end{align*}
The constraint eq.\ (\ref{(2,5) minimal model}) provides the knowledge of
\begin{align*}
P^{[1]}(x_2,x_2,x_3)
=&B_{2,2,2}\:x_2^4x_3^2\\
+&B_{2,2,1}(x_2^4x_3+2x_2^3x_3^2)\\
+&2B_{2,1,1}x_2^3x_3
+(B_{2,1,1}+2B_{2,2,0})x_2^2x_3^2
+B_{2,2,0}x_2^4\\
+&2B_{2,1,0}(x_2^3+x_2x_3^2)
+(B_{1,1,1}+2B_{2,1,0})\:x_2^2x_3\\
+&(B_{1,1,0}+2B_{2,0,0})x_2^2
+B_{2,0,0}x_3^2
+2B_{1,1,0}x_2x_3\\
+&B_{1,0,0}(2x_2+x_3)\\
+&B_{0,0,0},
\end{align*}
(obtained in the limit as $x_1\rechts x_2$), and thus of all $9$ coefficients.
So the three-point function is uniquely determined.
Since $\langle\left[T(x_1)T(x_2)\right]_{\reg}\rangle\:p_1p_2$ obtained from (\ref{definition of [T1T2]reg}) 
is just the $O(1)|_{x_1=x_2}$ part in the connected two-point function 
given by eq.\ (\ref{candidate for connected two-point function with correct finite behaviour}),
the remaining unknown constant in the two-point function is determined
using Example \ref{Example: Virasoro three-point function on X}.
\end{proof}

\textbf{Acknowledgement}

The author would like to thank W.\ Nahm for his help.

\appendix

\section{Proof of Theorem \ref{Theorem: Virasoro two-point function}}

\begin{enumerate}
\item 
We have
\begin{align*}
\langle T(x_1)T(x_2)\rangle\:&p_1p_2\\
&=\left[\langle T(x_1)T(x_2)\rangle\:p_1p_2\right]_{n-2}
+\left[\langle T(x_1)T(x_2)\rangle\:p_1p_2\right]_{\leq n-3},
\end{align*}
where according to (\ref{behaviour of T in infinity if n is odd}),
\begin{align*}
\left[\langle T(x_1)T(x_2)\rangle\:p_1p_2\right]_{n-2}
=&\:\frac{c}{32}a_0x^{n-2}\langle T(x_2)\rangle\:p_2\nn\\
=&\:\1^{-1}\left[\langle T(x_1)\rangle\langle T(x_2)\rangle\:p_1p_2\right]_{n-2},
\end{align*}
so
\begin{align*}
\langle T(x_1)T(x_2)\rangle\:&p_1p_2-\1^{-1}\langle T(x_1)\rangle\langle T(x_2)\rangle p_1p_2\\
=&\left[\langle T(x_1)T(x_2)\rangle\:p_1p_2\right]_{\leq n-3}
-\1^{-1}\left[\langle T(x_1)\rangle\langle T(x_2)\rangle p_1p_2\right]_{\leq n-3}.
\end{align*}
This shows (\ref{behaviour of Virasoro two-point function as x goes to infty}).
\item\label{Proof of the Theorem on the Virasoro-2-point function: Comparison with the OPE}
The proof is constructive. 
We build up a candidate and correct it subsequently so as to 
\begin{itemize}
\item 
match the singularities prescribed by the OPE, 
\item
behave at infinity according to (\ref{behaviour of Virasoro two-point function as x goes to infty}),
\item
be holomorphic in the appropriate coordinates away from the diagonal.
$X$ is covered by the coordinate patches $\{p\not=0\}$, $\{p'\not=0\}$ and $\{|x^{-1}|<\eps\}$.
\end{itemize}
General outline: The two-point function is meromorphic on $X$ whence rational. 
So once the singularities are fixed it is clear that we are left with the addition of polynomials as the only degree of freedom.
The key ingredient is the use of the rational function
\begin{align}\label{definition of the basic rational function f}
f(x_1,x_2):=\left(\frac{y_1+y_2}{x_1-x_2}\right)^2,
\end{align}
which has a double pole at $x_1=x_2$ as $y_1=y_2\not=0$, and is regular for $(x_1,y_1)$ close to $(x_2,-y_2)$.

For finite and fixed but generic $x_2$, and for the function $f$ defined by eq.\ (\ref{definition of the basic rational function f}),
we have
\begin{align*}
\frac{c}{32}\frac{1}{p_1p_2}f(x_1,x_2)^2
&=\frac{c/2}{(x_1-x_2)^4}+\frac{c}{16}\frac{[p_2']^2}{p_2^2(x_1-x_2)^2}+O(1),
\end{align*}
where $O(1)$ includes all terms regular at $x_1=x_2$.
Moreover,
\begin{align}\label{expansion of second order term}
\frac{1}{4}\frac{1}{p_1p_2}f(x_1,x_2)
=\frac{1}{p_2(x_1-x_2)^2} +O((x_1-x_2)^{-1})\:.
\end{align}
Thus
\begin{align*}
\frac{[p_2']^2}{p_2^2(x_1-x_2)^2}
\=\frac{1}{8p_1p_2}\left(\frac{[p_1']^2}{p_1}+\frac{[p_2']^2}{p_2}\right)f(x_1,x_2)
+O(1)\:.
\end{align*}
We conclude that
\begin{align}\label{leading term of OPE}
\frac{c/2}{(x_1-x_2)^4}\1
&=\frac{c}{32}\frac{1}{p_1p_2}\1\Big\{\:f(x_1,x_2)^2
\nn\\
&\quad\quad\quad\quad\quad\quad
-\frac{1}{4}f(x_1,x_2)\left(\frac{[p_1']^2}{p_1}+\frac{[p_2']^2}{p_2}\right)\Big\}+O(1).
\end{align}
Now by eq.\ (\ref{expansion of second order term}), 
\begin{align}\label{replacement of term propto (x1-x2) to power -2}
\frac{\langle T(x_1)\rangle+\langle T(x_2)\rangle}{(x_1-x_2)^2}
\=\frac{1}{4}f(x_1,x_2)
\left(\frac{\langle T(x_1)\rangle}{p_2}+\frac{\langle T(x_2)\rangle}{p_1}\right)+O(1)\:.
\end{align}
From eqs (\ref{leading term of OPE}) and (\ref{replacement of term propto (x1-x2) to power -2}) we obtain 
\begin{align*}
\frac{c/2}{(x_1-x_2)^4}&\1+\frac{\langle T(x_1)\rangle+\langle T(x_2)\rangle}{(x_1-x_2)^2}\\
\=\frac{1}{p_1p_2}\left(\frac{c}{32}\1f(x_1,x_2)^2+\frac{1}{16}f(x_1,x_2)(\Theta_1+\Theta_2)\right)+O(1)\:,
\end{align*}
by eq.\ (\ref{p<T(x)>})
Thus in the region where $x_1$ and $x_2$ are \textit{finite}, we have
\begin{align}
\1\langle T(x_1)T(x_2)\rangle_c\:p_1p_2
=R^{(0)}(x_1,x_2)+O\left(1\right)|_{x_1=x_2},
\label{candidate for connected two-point function with correct finite behaviour}
\end{align}
where 
\begin{align}\label{rational symmetric function}
R^{(0)}(x_1,x_2)
:=\:\frac{c}{32}f(x_1,x_2)^2\1
+\frac{1}{16}f(x_1,x_2)(\Theta_1+\Theta_2)\:.
\end{align}
Note that the $O(1)|_{x_1=x_2}$ terms are restricted to polynomials in $x_1,x_2$ and $y_1,y_2$.
This simplification is due to the use of the connected two-point function.

The degree requirement (\ref{behaviour of Virasoro two-point function as x goes to infty}) yields the upmost specification of 
eq.\ (\ref{candidate for connected two-point function with correct finite behaviour}),
because some terms appearing in
\begin{align*}
R^{(0)}(x_1,x_2)
&=\frac{c}{32}\1\frac{(p_1-p_2)^2}{(x_1-x_2)^4}\nn\\
&+\frac{c}{8}y_1y_2\1\frac{p_1+p_2}{(x_1-x_2)^4}
+\frac{c}{4}\1\frac{p_1p_2}{(x_1-x_2)^4}\nn\\
&+\frac{1}{16}\frac{p_1+2y_1y_2+p_2}{(x_1-x_2)^2}
\left(\Theta_1^{[1]}+\Theta_2^{[1]}\right)\nn\\
&+\frac{1}{16}\frac{p_1+2y_1y_2+p_2}{(x_1-x_2)^2}
\left(y_1\Theta_1^{[y]}+y_2\Theta_2^{[y]}\right)
\end{align*}
are absent in eq.\ (\ref{candidate for connected two-point function with correct finite behaviour})
and so determine some of the polynomials in the connected two-point function (which in the following we shall refer to as correction terms).
To keep formulas short, we shall go over to the rational function $R(x_1,x_2)$ 
introduced in part \ref{rational function of the two-point function} of Theorem \ref{Theorem: Virasoro two-point function},
since it has milder divergencies for $|x|$ large than $R^{(0)}(x_1,x_2)$ does. Thus we show now that
\begin{align}\label{reduction to the rational function of the claim}
R^{(0)}(x_1,x_2)
=R(x_1,x_2)+\text{polynomials},
\end{align}
where the ``polynomial`` part is a sum of polynomials $x_1,x_2$ and in $y_1,y_2$.
Indeed, we have the following identities:
\begin{align*}
\frac{(p_1-p_2)^2}{(x_1-x_2)^4}=\frac{p_1'p_2'}{(x_1-x_2)^2}+\text{polynomial}.
\end{align*}
Lemma \ref{Lemma expressions for the sum of peven and podd} yields
\begin{align}
\frac{c}{8}y_1y_2\1&\frac{p_1+p_2}{(x_1-x_2)^4}\nn\\ 
=\frac{c}{4}y_1y_2\1&\frac{p^{[1]}(\sqrt{x_1x_2})}{(x_1-x_2)^4} 
+\frac{c}{8}y_1y_2(x_1+x_2)\1\frac{p^{[x]}(\sqrt{x_1x_2})}{(x_1-x_2)^4}\nn\\
+\frac{c}{32}y_1y_2\1&\frac{\frac{1}{\sqrt{x_1x_2}}p^{[1]'}(\sqrt{x_1x_2})}{(x_1-x_2)^2}\nn\\
&\quad\quad\quad\quad
+\frac{3c}{64}y_1y_2(x_1+x_2)\1
\frac{\frac{1}{\sqrt{x_1x_2}}p^{[x]'}(\sqrt{x_1x_2})}{(x_1-x_2)^2}\nn\\
+\frac{c}{32}y_1y_2\1&\frac{p^{[1]''}(\sqrt{x_1x_2})}{(x_1-x_2)^2}\nn\\
&\quad\quad\quad\quad
+\frac{c}{64}y_1y_2(x_1+x_2)\1
\frac{p^{[x]''}(\sqrt{x_1x_2})}{(x_1-x_2)^2}\nn\\
+\text{polynomial}&.\label{reduction of bad powers, 4}
\end{align}
Likewise,
\begin{align}
\frac{1}{8}&y_1y_2\frac{\Theta_1^{[1]}+\Theta_2^{[1]}}{(x_1-x_2)^2}\nn\\
&=-\frac{c}{32}y_1y_2\left(\frac{\Pi^{[1]}(\sqrt{x_1x_2})}{(x_1-x_2)^2} 
+\frac{1}{2}(x_1+x_2)\frac{\Pi^{[x]}(\sqrt{x_1x_2})}{(x_1-x_2)^2}\right)\nn\\
&\quad+\text{polynomial}.\label{reduction of bad powers, 6}
\end{align}
Let $r,s$ be polynomials in the only one variable $x$. 
Then we have
\begin{align}
\frac{r_1s_1+r_2s_2}{(x_1-x_2)^2}&=\frac{r_1s_2+r_2s_1}{(x_1-x_2)^2}+\text{polynomial}\:\label{reduction of bad powers, 1}.
\end{align}
Thus
\begin{align}\label{p1 Theta2+p2 Theta1 over (x1-x2) to the -2}
\frac{1}{8}\frac{p_1\Theta_1+p_2\Theta_2}{(x_1-x_2)^2}
=\frac{1}{8}\frac{p_1\Theta_2+p_2\Theta_1}{(x_1-x_2)^2}+\text{polynomial}.
\end{align}
(\ref{reduction of bad powers, 1}) generalises to terms including $y_i$ as 
\begin{align*}
\frac{y_1r_1+y_2r_2}{(x_1-x_2)^2}
=\frac{y_1r_2+y_2r_1}{(x_1-x_2)^2}
+\frac{p_1-p_2}{x_1-x_2}\:\frac{r_1-r_2}{x_1-x_2}\:\frac{1}{y_1+y_2}\:.
\end{align*}
Thus
\begin{align*}
&\frac{1}{16}\frac{p_1+2y_1y_2+p_2}{(x_1-x_2)^2}(y_1\Theta_1^{[y]}+y_2\Theta_2^{[y]})\nn\\
=&\frac{1}{16}\frac{p_1+2y_1y_2+p_2}{(x_1-x_2)^2}\:(y_1\Theta_2^{[y]}+y_2\Theta_1^{[y]})+\text{polynomial},
\end{align*}
and Lemma \ref{Lemma expressions for the sum of peven and podd} yields 
\begin{align}
&\frac{1}{16}\frac{p_1+p_2}{(x_1-x_2)^2}\:\left(y_1\Theta_2^{[y]}+y_2\Theta_1^{[y]}\right) 
=\nn\\
&\frac{1}{8}\left(y_1\Theta_2^{[y]}+y_2\Theta_1^{[y]}\right)\times\nn\\
&\quad\quad\quad\quad\times\left(\frac{p^{[1]}(\sqrt{x_1x_2})}{(x_1-x_2)^2}
+\frac{1}{2}(x_1+x_2)\frac{p^{[x]}(\sqrt{x_1x_2})}{(x_1-x_2)^2}\right)\nn\\
&+\text{polynomial}.\label{reduction of bad powers, 5}
\end{align}
This proves eq.\ (\ref{reduction to the rational function of the claim}). 
Note that this result implies that in the finite region, 
$R(x_1,x_2)$ has the correct singularities. It remains to correct its behaviour for large $|x|$.
\item
We first subtract all terms from $R$ which are of non-admissible order in $x_1$.
These depend polynomially on $x_2$ because this is true for $\left[(x_1-x_2)^{-\ell}\right]_{>k}$ with $\ell\in\N$,
$k\in\Z$, ($x_1$ large), and may depend on $y_2$.   
The result may still be degree violating in $x_2$. 
Thus the corrected rational function reads 
\begin{align*}
&R-\left[R\right]_{>n-3}-\left[R-\left[R\right]_{>n-3}\right]^{>n-3}\\
=&R-\left[R\right]_{>n-3}-\left[R\right]^{>n-3}+\left[\left[R\right]_{>n-3}\right]^{>n-3}.
\end{align*}
Since the subtractions could be done in a different order,
the procedure only works due to  
\begin{align}\label{symmetry condition}
\left[\left[R\right]_{>n-3}\right]^{>n-3}=\left[\left[R\right]^{>n-3}\right]_{>n-3}.
\end{align}
The connected two-point function is thus determined up to addition of a polynomial $P(x_1,x_2,y_1,y_2)$
of the form (\ref{definition of the polynomial which is specific to the state}) which is specific to the state. 
The degree and symmetry requirements for $P(x_1,x_2,y_1,y_2)$ are immediate.

For clarity, we first list the terms contained in $-\left[R\right]_{>n-3}$ resp.\ $-y_1\left[R\right]_{>\frac{n}{2}-3}$:
From (\ref{reduction of bad powers, 4}),
\begin{align}
&-\frac{3c}{64}y_1y_2\left[x_1\1\frac{\frac{1}{\sqrt{x_1x_2}}p^{[x]'}(\sqrt{x_1x_2})}{(x_1-x_2)^2}\right]_{>\frac{n}{2}-3},
\label{correction line 3}\\
&-\frac{c}{64}y_1y_2\left[x_1\1\frac{p^{[x]''}(\sqrt{x_1x_2})}{(x_1-x_2)^2}\right]_{>\frac{n}{2}-3},
\label{correction line 9}
\end{align}
from (\ref{reduction of bad powers, 6}),
\begin{align}
&\frac{c}{64}y_1y_2\left[x_1\frac{\Pi^{[x]}(\sqrt{x_1x_2})}{(x_1-x_2)^2}\right]_{>\frac{n}{2}-3}\:,
\label{correction line 11}
\end{align}
from (\ref{p1 Theta2+p2 Theta1 over (x1-x2) to the -2}),
\begin{align}
&-\frac{1}{8}\Theta_2\left[\frac{p_1}{(x_1-x_2)^2}\right]_{>n-3}\:,
\label{correction line 4}
\end{align}
and from (\ref{reduction of bad powers, 5}),
\begin{align}
&\frac{1}{8}y_1\:\Theta_2^{[y]}\left[\frac{p^{[1]}(\sqrt{x_1x_2})}{(x_1-x_2)^2}\right]_{>\frac{n}{2}-3},
 \label{correction line 6}\\
&\frac{1}{16}\:y_1\Theta_2^{[y]}\left[x_1\frac{p^{[x]}(\sqrt{x_1x_2})}{(x_1-x_2)^2}\right]_{>\frac{n}{2}-3},
\label{correction line 7,1}\\
&\frac{1}{16}\:y_1x_2\Theta_2^{[y]}\left[\frac{p^{[x]}(\sqrt{x_1x_2})}{(x_1-x_2)^2}\right]_{>\frac{n}{2}-3}.
 \label{correction line 7,2}
\end{align}
Now we give the full explicit expression for
\begin{displaymath}
-\left[R\right]_{>n-3}-\left[R\right]^{>n-3}+\left[\left[R\right]_{>n-3}\right]^{>n-3}\:. 
\end{displaymath}
(\ref{correction line 3}) and (\ref{correction line 9}) yield
\begin{align*}
\frac{c}{64}y_1y_2(n^2-1)a_0x_1^{\frac{n}{2}-\frac{3}{2}}x_2^{\frac{n}{2}-\frac{5}{2}}
=-\frac{1}{8}y_1y_2A_0x_1^{\frac{n}{2}-\frac{3}{2}}x_2^{\frac{n}{2}-\frac{5}{2}}\:,
\end{align*}
which is cancels against the term we obtain from (\ref{correction line 11}).
For odd $n$, $A_0=-\frac{c}{8}(n^2-1)a_0\1$, 
so (\ref{correction line 4}) yields
\begin{align*}
-\frac{1}{8}a_0\left(x_1^{n-2}\Theta_2+x_2^{n-2}\Theta_1\right)
-\frac{c}{64}(n^2-1)a_0^2x_1^{n-2}x_2^{n-2}.
\end{align*}
(\ref{correction line 6}) yields
\begin{align*}
&\frac{1}{8}y_1a_1x_1^{\frac{n}{2}-\frac{5}{2}}x_2^{\frac{n}{2}-\frac{1}{2}}\Theta_2^{[y]}
+\frac{1}{8}y_2a_1x_1^{\frac{n}{2}-\frac{1}{2}}x_2^{\frac{n}{2}-\frac{5}{2}}\Theta_1^{[y]}\:.
\end{align*}
(\ref{correction line 7,1}) yields:
\begin{align*}
&\frac{1}{16}y_1a_0x_1^{\frac{n}{2}-\frac{3}{2}}x_2^{\frac{n}{2}-\frac{1}{2}}\Theta_2^{[y]}
+\frac{1}{16}y_2a_0x_1^{\frac{n}{2}-\frac{1}{2}}x_2^{\frac{n}{2}-\frac{3}{2}}\Theta_1^{[y]}\:,\\ 
 &\frac{1}{8}y_1a_0x_1^{\frac{n}{2}-\frac{5}{2}}x_2^{\frac{n}{2}+\frac{1}{2}}\Theta_2^{[y]}
 +\frac{1}{8}y_2a_0x_1^{\frac{n}{2}+\frac{1}{2}}x_2^{\frac{n}{2}-\frac{5}{2}}\Theta_1^{[y]}\:,\\ 
 &\frac{1}{16}y_1a_2x_1^{\frac{n}{2}-\frac{5}{2}}x_2^{\frac{n}{2}-\frac{3}{2}}\Theta_2^{[y]}
 +\frac{1}{16}y_2a_2x_1^{\frac{n}{2}-\frac{3}{2}}x_2^{\frac{n}{2}-\frac{5}{2}}\Theta_1^{[y]}\:.
\end{align*}
 (\ref{correction line 7,2}) yields:
 \begin{align*}
 &\frac{1}{16}y_1a_0x_1^{\frac{n}{2}-\frac{5}{2}}x_2^{\frac{n}{2}+\frac{1}{2}}\Theta_2^{[y]}
 +\frac{1}{16}y_2a_0x_1^{\frac{n}{2}+\frac{1}{2}}x_2^{\frac{n}{2}-\frac{5}{2}}\Theta_1^{[y]}.
 \end{align*}
\end{enumerate}
Since all terms are symmetric w.r.t.\ interchange of $x_1$ and $x_2$, 
eq.\ (\ref{symmetry condition}) has been verified.
This completes the proof.

\pagebreak

\end{document}